\documentclass[11pt]{amsart}%
\usepackage{amssymb}
\usepackage{amsmath}
\usepackage{amsfonts}
\usepackage{graphicx}%
\newtheorem{theorem}{Theorem}
\theoremstyle{plain}

\newtheorem{corollary}{Corollary}

\newtheorem{proposition}{Proposition}

\numberwithin{equation}{section}
\begin{document}
\title[Geometric Configuration of Riemannian Submanifolds...]{Geometric Configuration of Riemannian Submanifolds of arbitrary Codimension}
\author{Mohamed Abdelmalek}
\curraddr{Laboratoire Syst\`{e}mes Dynamiques et Applications\\
Universit\'{e} Aboubakr Belkaid, Tlemcen and Preparatory \\
School of Economy; Business and Management Sciences, \\
Tlemcen Algeria.}
\email{abdelmalekmhd@yahoo.fr}
\author{Mohammed Benalili}
\curraddr{Laboratoire Syet\`{e}mes Dynamiques et Applications \\
Dept. Mathematics, Faculty of Sciences\\
University Abou-bakr Belka\"{\i}d Tlemcen\\
Algeria.}
\email{m\_benalili@mail.univ-tlemcen.dz}
\author{Kamil Niedzia\l omski}
\curraddr{Department of Mathematics and Computer Science\\
University of Lodz\\
Banacha 22, 90-238 \L \'{o}d\'{z}\\
Poland.}
\email{kamiln@math.uni.lodz.pl}

\begin{abstract}
In this paper we study a geometric configuration of submanifolds of arbitrary
codimension in an ambient Riemannian space. We obtain relations between the
geometry of a $q$-codimension submanifold $M^{n}$ along its boundary and the
geometry of the boundary $\Sigma^{n-1}$of $M^{n}$ as an hypersuface of a
$q$-codimensional submanifold $P^{n}$ in an ambient space $\overline{M}^{n+q}%
$. As a consequence of these geometric ralations we get that the ellipticity
of the generalized Newton transformations implies the tranversality of $M^{n}$
and $P^{n}$ in $P^{n}$ is totally geodesic in $\overline{M}^{n+q}$.

\end{abstract}
\maketitle

\section{Introduction}

Let $\overline{M}^{n+q}$ be $n+q$-dimensional connected and orientable
Riemannian manifold with metric $\left\langle .,.\right\rangle $ and
Levi-Civita connection $\overline{\nabla}$. Denote by $P^{n}$ an oriented
connected $n$- submanifold of $\overline{M}^{n+q}$ and consider $\Sigma^{n-1}
$ an $n-1$-compact hypersurface of $P^{n}$. If $\Psi:M^{n}\rightarrow
\overline{M}^{n+q}$ is an oriented connected and compact submanifold of
$\overline{M}^{n+q}$ with boundary $\partial M$. $M^{n}$ will be said
submanifold of $\overline{M}^{n+q}$ with boundary $\Sigma^{n-1}$ if the
restriction of $\Psi$ to $\partial M$ is a diffeomorphism onto $\Sigma^{n-1}$.
A natural question would be: How can one describe the geometry of $M^{n}$
along its boundary $\partial M$ with respect to the geometry of the inclusion
$P^{n}\subset\overline{M}^{n+q}$? A partial answer to this question is given
by the following formula, obtained in this paper, which holds along the
boundary $\partial M$: for any multi-index $u=\left(  u_{1},...,u_{q}\right)
$ with length $\left\vert u\right\vert \leq n-1$,
\begin{equation}
\langle T_{u}\nu,\nu\rangle=\frac{1}{n-1-|u|}\sum_{l\leq u}\binom
{n-1-|l|}{\left\vert u\right\vert -l}\rho^{l}\mu^{u-l}\sigma_{|l|}(A_{\Sigma
}). \label{00}%
\end{equation}
where $\left(  T_{u}\right)  _{u}$ stands for the family of the generalized
Newton transformations introduced in \cite{4} associated to the matrix
$A=\left(  A_{1},...,A_{q}\right)  $; \ $\left(  A_{\alpha}\right)
_{\alpha\in\left\{  1,...,q\right\}  }$ is a system of matrices of the shape
operators corresponding to a normal basis to the manifold $\ M^{n}$ \ and
$\tilde{\sigma}_{u}=\tilde{\sigma}_{u}(A_{1}|_{\Sigma},...,A_{q}|_{\Sigma})$
are the coefficients of the Newton polynomial $P_{\widetilde{A}}:%
\mathbb{R}
^{q}\longrightarrow%
\mathbb{R}
$ defined by%
\[
P_{\widetilde{A}}(t)=\sum_{\left\vert u\right\vert \leq n-1}\tilde{\sigma}%
_{u}t^{u}%
\]
where $\widetilde{A}=(A_{1}|_{\Sigma},...,A_{q}|_{\Sigma})$ , \ $A_{\alpha
}|_{\Sigma}$ is the restriction of $A_{\alpha}$ to $\Sigma^{n-1}$ and
$\sigma_{r}(A_{\Sigma})$ is the symmetric function coefficient of $t^{r}$ in
the characteristic polynomial of the matrix $A_{\Sigma}$.

In \cite{2} Al\'{\i}as and Malacarne considered the above geometric
configuration to the case of hypersurfaces where $\Sigma_{n-1}$ is an $\left(
n-1\right)  $-dimensional compact submanifold contained in an hyperplane $\Pi$
of $%
\mathbb{R}
^{n+1}$ and $M^{n}$ stands for a smooth compact, connected and oriented
manifold with boundary $\partial M^{n}$. Moreover $M^{n}$ is an hypersurface
of $%
\mathbb{R}
^{n+1}$ with boundary $\Sigma_{n-1}$ in the sense that there exists
$\psi:M^{n}\rightarrow R^{n+1}$ an oriented hypersurface immersed in $%
\mathbb{R}
^{n+1}$ such that the restriction of $\psi$ to the boundary $\partial M^{n}$
is a diffeomorphism onto $\Sigma_{n-1}$. They showed that along the boundary
$\partial M^{n}$, for every $1\leq r\leq n-1$:%
\begin{equation}
\left\langle T_{r}\nu,\nu\right\rangle =\left(  -1\right)  ^{r}s_{r}%
\left\langle a,\nu\right\rangle ^{r}\label{01}%
\end{equation}
where $\nu$ stands for the outward pointing unit conormal vector field along
$\partial M^{n}$ while $T_{r}$ denotes the classical Newton transformation,
$a\in R^{n+1}$ such that $\Pi=a^{\bot}$ and $s_{r}$ is the $r$-th symmetric
function of the principal curvatures of $\Sigma_{n-1}$ with respect to $\nu$.
In \cite{3} Al\'{\i}as, de Lira and Malacarne studied the question in the
context of an $(n+1)$-dimensional connected oriented ambient Riemannian
manifold $\overline{M}^{n+1}$, they established that along the boundary
$\partial M^{n}$, for every $1\leq r\leq n-1$:%
\begin{equation}
\left\langle T_{r}\nu,\nu\right\rangle =\left(  -1\right)  ^{r}s_{r}%
\left\langle \xi,\nu\right\rangle ^{r}\label{02}%
\end{equation}
where $T_{r}$ , $\nu$, $s_{r}$, \ are as in relation $\;\not (\;$\ref{01})
where $P^{n}\subset\overline{M}^{n+1}$ is an embedded totally geodesic
submanifold instead of the hyperplane $\Pi$ and $\xi$ is a unitary normal
vector field to $P^{n}$. Relation (\ref{02}) shows that the ellipticity of the
Newton transformation $T_{r}$, for some $1\leq r\leq n-1$ on $M^{n}$, implies
the transversality of the hypersurfaces $M^{n}$ and $P^{n}$ along their
boundary. Formula (\ref{02}) was also obtained, in \cite{1}, by the two first
authors in context of pseudo-Riemannian spaces. Moreover we deduce from
relation (\ref{00}) that the ellipticity of the generalized Newton
transformation $T_{u}$ implies the transversality of the $q$ -codimension
submanifolds $M^{n}$ and $P^{n}$ in case where $P^{n}$ is totally geodesic
submanifold of $\overline{M}^{n+q}$.

\section{Preliminaries}

In this section, we will recall some properties of the generalized Newton
transformations and we will show how our method works.

\subsection{Generalized Newton Transformations}

Let $E$ be an $n$-dimensional real vector space and $End(E)$ be the vector
space of endomorphisms of $E$. Denote by $%
\mathbb{N}
$ the set of nonnegative integers and let $%
\mathbb{N}
^{q}$ be the one of multi- index $u=\left(  u_{1,}...,u_{q}\right)  $ with
$u_{j}\in%
\mathbb{N}
$. The length $\left\vert u\right\vert $ of $u$ is given by $\left\vert
u\right\vert =u_{1}+...+u_{q}$. $End^{q}\left(  E\right)  $ stands for the
vector space $End(E)\times...\times End(E)$ q-times. For $A=\left(
A_{1},\ldots,A_{q}\right)  \in End^{q}\left(  E\right)  $, $t=(t_{1}%
,...,t_{q})\in R^{q}$ and $u\in%
\mathbb{N}
^{q}$, we set%

\begin{align*}
tA  &  =t_{1}A_{1}+\ldots+t_{q}A_{q}\\
t^{u}  &  =t_{1}^{u_{1}}\ldots t_{q}^{u_{q}}\text{.}%
\end{align*}
For $\alpha\in\left\{  1,...,q\right\}  $ we define ( see \cite{4} ) the
musical functions $\alpha_{\flat}:%
\mathbb{N}
^{q}\longrightarrow%
\mathbb{N}
^{q}$ and $\alpha^{\sharp}:%
\mathbb{N}
^{q}\longrightarrow%
\mathbb{N}
^{q}$ by
\[
\alpha_{\flat}\left(  i_{1},\ldots,i_{q}\right)  =\left(  i_{1},\ldots
,i_{\alpha-1},i_{\alpha}-1,i_{\alpha+1},\ldots,i_{q}\right)
\]
and%
\[
\alpha^{\sharp}\left(  i_{1},\ldots,i_{q}\right)  =\left(  i_{1}%
,\ldots,i_{\alpha-1},i_{\alpha}+1,i_{\alpha+1},\ldots,i_{q}\right)
\]
It is clear that $\alpha_{\flat}$ is the inverse map of $\alpha^{\sharp}$.

The generalized Newton transformation (GNT in brief ) is a system of
endomorphisms $T_{u}=T_{u}\left(  A\right)  ,$ $u\in%
\mathbb{N}
^{q}$, that satisfies the following recursive relations
\begin{align*}
T_{0}  &  =I &  &  \text{where $0=(0,\ldots,0)$},\\
T_{u}  &  =\sigma_{u}I-\underset{\alpha}{\sum}A_{\alpha}T_{\alpha_{\flat
}\left(  u\right)  } & \text{where }\left\vert u\right\vert  &  >1\\
&  =\sigma_{u}I-\underset{\alpha}{\sum}T_{\alpha_{\flat}\left(  u\right)
}A_{\alpha} &  &
\end{align*}
where $\sigma_{u}$ are the coefficients of the Newton polynomial $P_{A}:%
\mathbb{R}
^{q}\longrightarrow%
\mathbb{R}
$ of $A$, given by%
\[
P_{A}(t)=\det\left(  I+tA\right)  =\underset{\left\vert u\right\vert \leq
n}{\sum\sigma_{u}t^{u}}%
\]
$\sigma_{u}=\sigma_{u}(A_{1},...,A_{q})$ depends only on $A=\left(
A_{1},...,A_{q}\right)  $ and $I$ is the identity map on $E$.

\subsection{The method}

We will describe how our method works.

\subsubsection{Hypersurfaces' case}

Let $M^{n}$ be a $n$-submanifold of codimension one in $\overline{M}^{n+1}$ of
boundary $\partial M$. Assume the boundary $\Sigma^{n-1}=\partial M$ \ is a
codimension one in $P^{n}\subset\overline{M}^{n+1}$. Then we have the
inclusions
\[
\Sigma^{n-1}\subset M^{n}\subset\overline{M}^{n+1}\text{, }\Sigma^{n-1}\subset
P^{n}\subset\overline{M}^{n+1}\text{.}%
\]
Denote the corresponding shape operators, respectively, by%
\[
A_{\Sigma}\text{, \ }A_{P}\text{, \ }A_{\Sigma,P}\text{, \ }A\text{.}%
\]
In our consideration we will need only $A_{\Sigma}$, $\ \ A_{P}$, $\ A$. More
precisely we will use
\[
A_{\Sigma}\text{, \ \ }A_{P}|_{\Sigma}\text{, \ }A\text{.}%
\]
First two are represented by square matrices of dimension $n-1$ whereas the
last one by a square matrix of dimension $n$. The intrinsic geometry of
$\ \Sigma^{n-1}$ in $M^{n}$ is coded in the pair ($A_{\Sigma},A_{P}|_{\Sigma}
$) and the geometry of $M^{n}\subset\overline{M}^{n+1}$is given by $A$.
Therefore we will use the following Newton Transformation and the generalized
Newton Transformations%
\[
T_{\left(  k,l\right)  }=T_{\left(  k,l\right)  }\left(  A_{\Sigma}%
,A_{P}|_{\Sigma}\right)  \text{ and }T_{r}=T_{r}\left(  A\right)
\]
and corresponding symmetric functions
\[
\sigma_{\left(  k,l\right)  }=\sigma_{\left(  k,l\right)  }\left(  A_{\Sigma
},A_{P}|_{\Sigma}\right)  \text{ and }\sigma_{r}=\sigma_{r}\left(  A\right)
\text{.}%
\]
The goal is to show that%
\begin{equation}
\left\langle T_{r}\nu,\nu\right\rangle =\sum_{k+l=r}\sigma_{\left(
k,l\right)  } \label{1''}%
\end{equation}
where $\nu$ is the unit normal vector to $\Sigma^{n-1}$ in $M^{n}.$

The only geometric considerations involved are the ones which lead to the
formulas
\[
\left\langle N,\nu\right\rangle =\left\langle \xi,N\right\rangle \text{,
}\left\langle \eta,N\right\rangle =-\left\langle \xi,\nu\right\rangle
\]
and
\[
\left\langle Ae_{i},e_{j}\right\rangle =-\left\langle A_{\Sigma}e_{i}%
,e_{j}\right\rangle \left\langle \xi,\nu\right\rangle +\left\langle A_{P}%
e_{i},e_{j}\right\rangle \left\langle \xi,N\right\rangle
\]
where $N$ is unit normal vector with the respect to inclusion $M^{n}%
\subset\overline{M}^{n+1}$, $\xi$ unit normal vector with respect to
$P^{n}\subset$ $\overline{M}^{n+1}$and $\eta$ is the unit normal vector of
$\Sigma^{n-1}\subset P^{n}$ and $\left(  e_{1},...,e_{n-1}\right)  $ is a
local orthonormal basis of $T\Sigma^{n-1}$, we may assume that this basis
consists of eigenvectors of $A_{\Sigma}$ i.e. $A_{\Sigma}e_{i}=\tau_{i}e_{i}$.
In other words%
\[
A|_{\Sigma}=-\left\langle \xi,\nu\right\rangle A_{\Sigma}+\left\langle
\xi,N\right\rangle A_{P}|_{\Sigma}\text{.}%
\]
Assuming $P$ is totally umbilical in $\overline{M}^{n+1}$ , we have
$A_{P}|_{\Sigma}=\lambda I_{T\Sigma^{n-1}}$. Hence
\begin{equation}
A|_{\Sigma}=-\left\langle \xi,\nu\right\rangle A_{\Sigma}+\lambda\left\langle
\xi,N\right\rangle I_{T\Sigma^{n-1}}\text{.} \label{2''}%
\end{equation}

Denote by $\widetilde{A}$ the matrix of $A$ with the respect of the basis
$\left(  e_{1},...,e_{n-1},\nu\right)  $ and by $A$ the matrix of $A|_{\Sigma}
$. Then%
\[
\widetilde{A}=\left(
\begin{array}
[c]{cc}%
A & B\\
B^{\intercal} & c
\end{array}
\right)  \text{ , where }B=\left(
\begin{array}
[c]{c}%
\left\langle A\nu,e_{1}\right\rangle \\
\vdots\\
\left\langle A\nu,e_{n-1}\right\rangle
\end{array}
\right)  \text{ and }c=\left\langle A\nu,\nu\right\rangle \text{.}%
\]
Let us compare symmetric functions of $\widehat{A}$ with symmetric functions
of $A$. We have%
\begin{align*}
P_{\widetilde{A}}(t)  &  =\det\left(
\begin{array}
[c]{cc}%
I_{n-1}+tA & tB\\
tB^{\intercal} & 1+tc
\end{array}
\right)  =(1+tc-t^{2}B^{\intercal}(I_{n-1}+tA)^{-1}B)\det(I_{n-1}+tA)\\
&  =f(t)P_{A}(t)\text{, }%
\end{align*}
where $f(t)=1+tc-t^{2}B^{\intercal}(I_{n-1}+tA)^{-1}B$ . Recall that%
\[
P_{\widetilde{A}}(t)=\sum_{j=0}^{n}(-1)^{j}\sigma_{j}(\widetilde{A}%
)t^{j}\text{.}%
\]
Hence
\begin{align*}
(-1)^{r}r!\sigma_{r}(\widetilde{A})  &  =\frac{d^{r}}{dt^{r}}P_{\widetilde{A}%
}(0)=%
{\displaystyle\sum\limits_{j=0}^{r}}
\left(
\begin{array}
[c]{c}%
r\\
j
\end{array}
\right)  \frac{d^{r-j}}{dt^{r-j}}P_{A}^{\left(  r-j\right)  }(0)f^{\left(
j\right)  }(0)\\
&  =%
{\displaystyle\sum\limits_{j=0}^{r}}
\left(
\begin{array}
[c]{c}%
r\\
j
\end{array}
\right)  \left(  -1\right)  ^{r-j}(r-j)!\sigma_{r-j}\left(  A\right)
f^{\left(  j\right)  }\left(  0\right)  .
\end{align*}
It is not hard to see that%
\[
f(0)=1\text{, }f^{\prime}(0)=-c\text{ \ and }f^{\left(  j\right)
}(0)=-j!B^{\intercal}A^{j-2}B\text{ for }j\geq2\text{.}%
\]
Therefore
\begin{equation}
\sigma_{r}(\widetilde{A})=\sigma_{r}\left(  A\right)  +c\sigma_{r-1}\left(
A\right)  -\sum_{j=2}^{r}(-1)^{j}\left(  B^{\intercal}A^{j-2}B\right)
\sigma_{r-j}\left(  A\right)  . \label{4'}%
\end{equation}
Let us now move to symmetric functions of two matrices. Notice first that
\[
\sigma_{r}(\widetilde{A}+\lambda I_{n})=%
{\displaystyle\sum\limits_{j=0}^{r}}
\left(
\begin{array}
[c]{c}%
n-j\\
r-j
\end{array}
\right)  \lambda^{r-j}\sigma_{j}\left(  \widetilde{A}\right)  \text{.}%
\]
Indeed, $P_{\widetilde{A}+\lambda I_{n}}(t)=(1+t(a_{1}+\lambda)...(1+t(a_{n}%
+\lambda)$ if $a_{1},...,a_{n}$ are the eigenvalues of $\widetilde{A}$. Notice
moreover that%
\begin{align*}
P_{\widetilde{A},\lambda I_{n}}(t)  &  =\det\left(  I_{n}+t\widetilde
{A}+s\lambda I_{n}\right) \\
&  =\left(  1+ta_{1}+s\lambda\right)  ...\left(  1+ta_{n}+s\lambda\right)
\text{.}%
\end{align*}
Thus
\begin{equation}
\sigma_{(k,l)}(\widetilde{A},\lambda I_{n})=\left(
\begin{array}
[c]{c}%
n-k\\
l
\end{array}
\right)  \lambda^{l}\sigma_{k}(\widetilde{A})\text{.} \label{5'}%
\end{equation}
Hence%
\begin{equation}
\sigma_{r}(\widetilde{A}+\lambda I_{n})=\sum_{j=0}^{r}\sigma_{(j,r-j)}%
(\widetilde{A},\lambda I_{n}) \label{6'}%
\end{equation}
We need to show, by (\ref{6'}) and (\ref{1''}), that%
\begin{equation}
\left\langle T_{r}\nu,\nu\right\rangle =\widetilde{\sigma}_{r}\left(
A\right)  \text{.} \label{3'}%
\end{equation}
By the recurrence formula for $T_{r}$ we have%
\begin{align*}
\left\langle T_{r}\nu,\nu\right\rangle  &  =\sigma_{r}\left(  \widetilde
{A}\right)  -\left\langle T_{r-1}\nu,\widetilde{A}\nu\right\rangle \\
&  =\sigma_{r}\left(  \widetilde{A}\right)  -c\left\langle T_{r-1}\nu
,\nu\right\rangle -\sum_{i=1}^{n-1}\left\langle T_{r-1}\nu,e_{i}\right\rangle
\left\langle \widetilde{A}\nu,e_{i}\right\rangle .
\end{align*}
We will show that%

\begin{equation}
\sum_{i=1}^{n-1}\left\langle T_{k}\nu,e_{i}\right\rangle \left\langle
\widetilde{A}\nu,e_{i}\right\rangle =-\sum_{j=2}^{k}\left(  -1\right)
^{j}\left(  B^{\intercal}A^{j-2}B\right)  \sigma_{k-j}(A). \label{3''}%
\end{equation}
Indeed, assuming $(e_{i})$ is a basis consisting of eigenvectors with
eigenvalues $\left(  \tau_{i}\right)  _{i}$ by induction we have
\begin{align*}
\sum_{i=1}^{n-1}\left\langle T_{k}\nu,e_{i}\right\rangle \left\langle
\widetilde{A}\nu,e_{i}\right\rangle  &  =-b_{i}^{2}\sum_{j=2}^{k}\left(
-j\right)  \sigma_{k-j}\left(  A\right)  \tau_{i}^{j-2}\\
&  =\sum_{j=2}^{k}(-1)^{j}\left(  b_{i}\tau_{i}^{j-2}b_{i}\right)
\sigma_{k-j}\left(  A\right)
\end{align*}
where $B^{\intercal}=(b_{i})_{i}$, which proves (\ref{3''}). Thus by induction
and (\ref{4'}), we get%
\[
\left\langle T_{r}\nu,\nu\right\rangle =\sigma_{r}(\widetilde{A}%
)-c\sigma_{r-1}(A)+\sum_{j=2}^{r}\left(  -1\right)  ^{j}\left(  B^{\intercal
}A^{j-2}B\right)  \sigma_{r-j}(A)=\widetilde{\sigma}_{r}(A)\text{.}%
\]

\subsubsection{General case}

We generalize considerations to any codimension $q$. We generalize equations
(\ref{1''}) and (\ref{3'}). Let $P^{n}$ and $M^{n}$ be of codimension $q$ in
$\overline{M}^{n+q}$. We assume that $P^{n}$ is totally umbilical in
$\overline{M}^{n+1}$. As before $\Sigma^{n-1}\subset P^{n}$ is a boundary of
$M^{n}$. Then we have the following shape operators
\[
A_{\Sigma}\text{, \ }A_{P}^{\xi_{1}},...,A_{P}^{\xi_{p}}\text{, \ }A^{N_{1}%
},...,A^{N_{q}},
\]
corresponding to inclusions $\Sigma^{n-1}\subset P^{n}$, $P^{n}\subset
\overline{M}^{n+q}$, and $M^{n}\subset\overline{M}^{n+q}$, where $\left(
\xi_{1},...,\xi_{q}\right)  $ are orthogonal to $P^{n}$ and $\left(
N_{1},...,N_{q}\right)  $ are orthogonal to $M^{n}$. Let
\[
T_{u}=T_{u}\left(  A^{N_{1}},...,A^{N_{q}}\right)  \text{,\ \ \ }T_{v}%
=T_{v}\left(  A_{\Sigma},\text{ }A^{\xi_{1}}|_{\Sigma},...,A^{\xi_{q}%
}|_{\Sigma}\right)
\]
and%
\[
\widetilde{T_{u}}=\widetilde{T}_{u}\left(  A^{N_{1}}|_{\Sigma},...,A^{N_{q}%
}|_{\Sigma}\right)
\]
where $u$ is of length $q$ and $v$ of length $q+1$. We hope that the following
relations hold%

\begin{equation}
\left\langle T_{u}\nu;\nu\right\rangle =\widetilde{\sigma}_{u} \label{3'''}%
\end{equation}
and%
\[
\widetilde{\sigma}_{u}=%
{\displaystyle\sum\limits_{\left\vert v\right\vert =\left\vert u\right\vert
\text{, }v\text{ not increasing }}}
c_{v}\sigma_{v}%
\]
where $c_{v}$ is a coefficient independent of $\left(  A^{N_{1}}|_{\Sigma
},...,A^{N_{q}}|_{\Sigma}\right)  $. Both of them imply%
\[
\left\langle T_{u}\nu;\nu\right\rangle =%
{\displaystyle\sum\limits_{\left\vert v\right\vert =\left\vert u\right\vert
\text{, }v\text{ not increasing }}}
c_{v}\sigma_{v}.
\]

The first equality is a generalization of (\ref{3'}), whereas the second one
of (\ref{6'}). Moreover, the first one is purely algebraic and the second one
uses correspondence between $A_{\Sigma},$ $\ A^{\xi_{1}}|_{\Sigma}%
,...,A^{\xi_{q}}|_{\Sigma}$ and $A^{N_{1}}|_{\Sigma},...,A^{N_{q}}|_{\Sigma}$
analogous to the relation (\ref{2''}) in the codimension one case.

Let us be more precise. First notice that

\bigskip%
\begin{align*}
\overline{\nabla}_{e_{i}}e_{j}  &  =\sum_{k=1}^{n-1}\langle\overline{\nabla
}_{e_{i}}e_{j},e_{k}\rangle e_{k}+\langle\overline{\nabla}_{e_{i}}e_{j}%
,\nu\rangle\nu+\sum_{\alpha=1}^{q}\langle\overline{\nabla}_{e_{i}}%
e_{j},N_{\alpha}\rangle N_{\alpha}\\
&  =\sum_{k=1}^{n-1}\langle\overline{\nabla}_{e_{i}}e_{j},e_{k}\rangle
e_{k}+\langle\overline{\nabla}_{e_{i}}e_{j},\nu\rangle\nu+\sum_{\alpha=1}%
^{q}\langle A^{N_{\alpha}}e_{i},e_{j}\rangle N_{\alpha}%
\end{align*}
and
\begin{align*}
\overline{\nabla}_{e_{i}}e_{j}  &  =\sum_{k=1}^{n-1}\langle\overline{\nabla
}_{e_{i}}e_{j},e_{k}\rangle e_{k}+\langle\overline{\nabla}_{e_{i}}e_{j}%
,\eta\rangle\eta+\sum_{\alpha=1}^{q}\langle\overline{\nabla}_{e_{i}}e_{j}%
,\xi_{\alpha}\rangle\xi_{\alpha}\\
&  =\sum_{k=1}^{n-1}\langle\overline{\nabla}_{e_{i}}e_{j},e_{k}\rangle
e_{k}+\langle A_{\Sigma}e_{i},e_{j}\rangle\eta+\sum_{\alpha=1}^{q}\langle
A^{\xi_{\alpha}}e_{i},e_{j}\rangle\xi_{\alpha}%
\end{align*}
Thus
\[
\langle\overline{\nabla}_{e_{i}}e_{j},\nu\rangle\nu+\sum_{\alpha=1}^{q}\langle
A^{N_{\alpha}}e_{i},e_{j}\rangle N_{\alpha}=\ \langle A_{\Sigma}e_{i}%
,e_{j}\rangle\eta+\sum_{\alpha=1}^{q}\langle A^{\xi_{\alpha}}e_{i}%
,e_{j}\rangle\xi_{\alpha}%
\]
Hence%
\begin{align*}
\left\langle A^{N_{\alpha}}e_{i},e_{j}\right\rangle  &  =\left\langle
\eta,N_{\alpha}\right\rangle \left\langle A_{\Sigma}(e_{i}),e_{j}\right\rangle
+\sum_{\beta=1}^{q}\left\langle \xi_{\beta},N_{\alpha}\right\rangle
\langle\left(  A^{\xi_{\beta}}\right)  e_{i},e_{j}\rangle\\
&  =\left\langle \eta,N_{\alpha}\right\rangle \left\langle A_{\Sigma}%
(e_{i}),e_{j}\right\rangle +\sum_{\beta=1}^{q}\left\langle \xi_{\beta
},N_{\alpha}\right\rangle \langle A^{\xi_{\beta}}e_{i},e_{j}\rangle
\end{align*}

Assuming that $P^{n}$ is totally umbilical, i.e. $A^{\xi_{\alpha}}%
=\lambda_{\alpha}I_{n}$, where $I_{n}$ denotes the identity map of the tangent
space $T_{p}P^{n}$, we get%
\[
A^{N_{\alpha}}|_{\Sigma}=\left\langle \eta,N_{\alpha}\right\rangle A_{\Sigma
}+\sum_{\beta=1}^{q}\left\langle \xi_{\beta},N_{\alpha}\right\rangle
\lambda_{\beta}I_{n-1}\text{.}%
\]
Notice that it can be written in the form%
\begin{equation}
A^{N_{\alpha}}|_{\Sigma}=\left\langle \eta,N_{\alpha}\right\rangle A_{\Sigma
}+\left\langle V,N_{\alpha}\right\rangle I_{n-1}\text{, \ \ where }%
V=\sum_{\beta=1}^{q}\lambda_{\beta}\xi_{\beta}. \label{7'}%
\end{equation}
Moreover one can show that%
\[
\left\langle \eta,V\right\rangle =\det\left(  \left\langle \xi_{\alpha
},N_{\beta}\right\rangle \right)  _{\alpha,\beta}\text{ \ and }\left\langle
\eta,N_{\alpha}\right\rangle =-\det C_{\alpha}%
\]
where $C_{\alpha}$ is a matrix obtained from $C=\left(  \left\langle
\xi_{\alpha},N_{\beta}\right\rangle \right)  _{\alpha,\beta}$ by replacing the
$\alpha$-th column by $\left\langle \xi_{\alpha},\nu\right\rangle $ .

Hence we have%

\[
\widetilde{T}_{u}=\widetilde{T}_{u}\left(  \rho_{1}A_{\Sigma}+\mu_{1}%
I_{n-1},...,\rho_{q}A_{\Sigma}+\mu_{q}I_{n-1}\right)  \text{, where \ \ }%
\rho_{\alpha}=\left\langle \eta,N_{\alpha}\right\rangle \text{, }\mu_{\alpha
}=\left\langle V,N_{\alpha}\right\rangle \text{.}%
\]
Now if $\left\{  e_{1},\ldots,e_{n-1}\right\}  $ is a basis of eigenvectors of
the shape operator $A_{\Sigma}$ then, for every $i\in\left\{  1,\ldots
,n-1\right\}  $,
\[
A_{\Sigma}(e_{i})=\tau_{i}e_{i}.
\]
where $\tau_{i}$ are the corresponding eigenvalues.

By relation (\ref{7'}) we get
\begin{align*}
\langle A_{\alpha}(e_{i}),e_{j}\rangle &  =\left(  \left\langle \eta
,N_{\alpha}\right\rangle \tau_{i}+\left\langle V,N_{\alpha}\right\rangle
\right)  \delta_{i}^{j}\\
&  =\gamma_{i,\alpha}\delta_{i}^{j},
\end{align*}
where
\[
\gamma_{i,\alpha}=\left\langle \eta,N_{\alpha}\right\rangle \tau
_{i}+\left\langle V,N_{\alpha}\right\rangle .
\]

Thus the matrix associated to the shape operator $A_{\alpha}$ with respect to
the basis $\left\{  e_{1},\ldots,e_{n-1},\nu\right\}  $ is given by
\[
A_{\alpha}=\left(
\begin{array}
[c]{ccccc}%
\gamma_{1,\alpha} & 0 & \hdots & 0 & \langle A_{\alpha}\nu,e_{1}\rangle\\
0 & \gamma_{2,\alpha} & \hdots & 0 & \vdots\\
\vdots & \vdots & \ddots & \vdots & \vdots\\
0 & \hdots & \hdots & \gamma_{n-1,\alpha} & \langle A_{\alpha}\nu
,e_{n-1}\rangle\\
\langle A_{\alpha}\nu,e_{1}\rangle & \hdots & \hdots & \langle A_{\alpha}%
\nu,e_{n-1}\rangle & \langle A_{\alpha}\nu,\nu\rangle
\end{array}
\right)  \text{.}%
\]
.

\section{Algebraic Formulas}

Let us prove the algebraic relation (\ref{3'''}), to do so, we put
\[
\rho_{\alpha}=\langle\eta,N_{\alpha}\rangle,\quad\mu_{\alpha}=\langle
V,N_{\alpha}\rangle.
\]
Moreover, we will write $v\leq u$ for multi--indices $v,u\in\mathbb{N}^{q}$ if
the difference $u-v\in\mathbb{N}^{q}$. For a multi--index $u=\left(
u_{1},...,u_{q}\right)  \in\mathbb{N}^{q}$ of length $\left\vert u\right\vert
=n$, we set
\[
\binom{n}{u}=\frac{n!}{u!}=\frac{n!}{u_{1}!\ldots u_{q}!}.
\]

\begin{proposition}
\label{Prop1} \bigskip Let $\tilde{A}=(A_{1}|_{\Sigma},...,A_{q}|_{\Sigma
})=(\rho_{1}A_{\Sigma}+\mu_{1}I,\ldots,\rho_{q}A_{\Sigma}+\mu_{q}I)$. Put
$\tilde{\sigma}_{u}=\sigma_{u}(\tilde{A})$. Then for every multi-index $u\in%
\mathbb{N}
^{q}$ we have
\begin{equation}
\tilde{\sigma}_{u}=\frac{1}{\left(  n-1-\left\vert u\right\vert \right)  !}%
{\displaystyle\sum\limits_{l\leq u}}
\left(
\begin{array}
[c]{c}%
\left\vert l\right\vert \\
l
\end{array}
\right)  \rho^{l}\mu^{u-l}\left(
\begin{array}
[c]{c}%
n-1-\left\vert l\right\vert \\
u-l
\end{array}
\right)  \sigma_{\left\vert l\right\vert }\left(  A_{\Sigma}\right)  .
\label{1'}%
\end{equation}

\end{proposition}

\begin{proof}
First, by applying the following relation (see \cite{6})
\[
\sigma_{u}(aA_{1},\ldots,A_{q})=a^{u_{1}}\sigma_{u}(A_{1},\ldots,A_{q})
\]
we obtain%
\begin{equation}
\tilde{\sigma}_{u}(\rho_{1}A_{\Sigma}+\mu_{1}I,\ldots,\rho_{q}A_{\Sigma}%
+\mu_{q}I)=\rho^{u}\hat{\sigma}_{u}(A_{\Sigma}+\theta_{1}I,\ldots,A_{\Sigma
}+\theta_{q}I) \label{2'}%
\end{equation}
where%
\[
\theta_{\alpha}=\frac{\mu_{\alpha}}{\rho_{\alpha}}\text{ and \ }\rho^{u}%
=\rho_{1}^{u_{1}}...\rho_{q}^{u_{q}}.
\]
Consider the GNT $T_{u}=T_{u}(\hat{A})$, with%
\[
\hat{A}=(A_{\Sigma}+\theta_{1}I,\ldots,A_{\Sigma}+\theta_{q}I).
\]
The characteristic polynomial associate to $\hat{A}$ is then%
\[
P_{\hat{A}}(t)=\underset{\left\vert u\right\vert \leq n-1}{\sum}\hat{\sigma
}_{u}t^{u}%
\]
Moreover, we have
\begin{align}
P_{\hat{A}}(t)  &  =\det\left(  I+\sum_{\alpha=1}^{q}t_{\alpha}\left(
A_{\Sigma}+\theta_{\alpha}\right)  I\right) \\
&  =%
{\displaystyle\prod\limits_{j=1}^{n-1}}
\left(  1+\tau_{j}\left(  t_{1}+..+t_{q}\right)  +t_{1}\theta_{1}+\ldots
+t_{q}\theta_{q}\right)  .
\end{align}
By expanding the linear factorization of a monic polynomial we obtain
\[
P_{\hat{A}}(t)=\sum_{j=0}^{n-1}\left(  1+t_{1}\theta_{1}+...+t_{q}\theta
_{q}\right)  ^{n-1-j}\left(  t_{1}+...+t_{q}\right)  ^{j}\sigma_{j}\left(
A_{\Sigma}\right)
\]
and by the multinomial theorem we get%
\[
P_{\hat{A}}(t)=\sum_{j=0}^{n-1}\sum_{k=0}^{n-1-j}\sum_{\left\vert v\right\vert
=k}\sum_{\left\vert l\right\vert =j}\frac{1}{\left(  n-1-j-k\right)  !}\left(
\begin{array}
[c]{c}%
j\\
l
\end{array}
\right)  \left(
\begin{array}
[c]{c}%
n-1-j\\
v
\end{array}
\right)  \sigma_{j}\left(  A_{\Sigma}\right)  \theta^{v}t^{v+l}.
\]
where $l=(l_{1},...,l_{q})$, $v=\left(  v_{1},...,v_{q}\right)  $,
$\theta=\left(  \theta_{1},...,\theta_{q}\right)  $ and $t=\left(
t_{1},...,t_{q}\right)  $. Let $u=v+l$ so $j=\left\vert u\right\vert
-\left\vert v\right\vert $ and%
\[
P_{\hat{A}}(t)=\sum_{\left\vert u\right\vert \leq n-1}\sum_{l\leq u}\frac
{1}{\left(  n-1-\left\vert u\right\vert \right)  !}\left(
\begin{array}
[c]{c}%
\left\vert l\right\vert \\
l
\end{array}
\right)  \left(
\begin{array}
[c]{c}%
n-1-\left\vert l\right\vert \\
u-l
\end{array}
\right)  \theta^{u-l}\sigma_{\left\vert l\right\vert }\left(  A_{\Sigma
}\right)  t^{u}\text{.}%
\]
Thus the coefficient before $t^{u}$ is given by%
\[
\hat{\sigma}_{u}(A_{\Sigma}+\theta_{1}I,\ldots,A_{\Sigma}+\theta_{q}%
I)=\frac{1}{\left(  n-1-\left\vert u\right\vert \right)  !}%
{\displaystyle\sum\limits_{l\leq u}}
\left(
\begin{array}
[c]{c}%
\left\vert l\right\vert \\
l
\end{array}
\right)  \theta^{u-l}\left(
\begin{array}
[c]{c}%
n-1-\left\vert l\right\vert \\
u-l
\end{array}
\right)  \sigma_{\left\vert l\right\vert }\left(  A_{\Sigma}\right)  \text{.}%
\]
Consequently%
\[
\tilde{\sigma}_{u}(\rho_{1}A_{\Sigma}+\mu_{1}I,\ldots,\rho_{q}A_{\Sigma}%
+\mu_{q}I)=
\]

\begin{equation}
\frac{1}{\left(  n-1-\left\vert u\right\vert \right)  !}%
{\displaystyle\sum\limits_{l\leq u}}
\left(
\begin{array}
[c]{c}%
\left\vert l\right\vert \\
l
\end{array}
\right)  \rho^{l}\mu^{u-l}\left(
\begin{array}
[c]{c}%
n-1-\left\vert l\right\vert \\
u-l
\end{array}
\right)  \sigma_{\left\vert l\right\vert }\left(  A_{\Sigma}\right)
\label{2'''}%
\end{equation}

On the other hand

Denote by $\bar{A}=(A_{\Sigma},\mu_{1}I,\ldots,\mu_{q}I)$ and $\bar{\sigma
}_{\left(  k,v\right)  }=\sigma_{\left(  k,v\right)  }(\bar{A})$. We have
\[
\sigma_{(k,v)}(\bar{A})=\mu^{v}\sigma_{\left(  k,v\right)  }(A_{\Sigma
},I,\ldots,I).
\]
Now, we use the following formula ( see \cite{6} )
\[
\sigma_{(}v,k)(B_{1},\ldots,B_{q},I)=\binom{n-1-|v|}{k}\sigma_{v}(B_{1}%
,\ldots,B_{q})
\]
consecutively $q$-times and we get
\begin{align*}
\sigma_{(j,v)}(A_{\Sigma},I,\ldots,I)  &  =\binom{n-1-|v|+v_{1}}{v_{1}}%
\ldots\binom{n-1-|v|+|v|}{v_{q}}\sigma_{j}(A_{\Sigma})\\
&  =\frac{1}{\left(  n-1-|v|-j\right)  !}\binom{n-1-j}{v}\sigma_{j}(A_{\Sigma
}).
\end{align*}
Taking $j=|u|-|v|$, we obtain%
\[
=\frac{1}{\left(  n-1-|u|\right)  !}\binom{n-1-|u|+|v|}{v}\sigma
_{|u|-|v|}(A_{\Sigma})
\]
and if we let $l=u-v,$we have%
\begin{equation}
\sigma_{(|l|,u-l)}(A_{\Sigma},I,\ldots,I)=\frac{1}{\left(  n-1-|u|\right)
!}\binom{n-1-|l|}{u-l}\sigma_{|l|}(A_{\Sigma}) \label{3}%
\end{equation}
So by relations (\ref{2'}) and (\ref{3}), we get%
\begin{equation}
\tilde{\sigma}_{u}=\frac{1}{\left(  n-1-\left\vert u\right\vert \right)  !}%
{\displaystyle\sum\limits_{l\leq u}}
\left(
\begin{array}
[c]{c}%
\left\vert l\right\vert \\
l
\end{array}
\right)  \rho^{l}\mu^{u-l}\left(
\begin{array}
[c]{c}%
n-1-\left\vert l\right\vert \\
u-l
\end{array}
\right)  \sigma_{\left\vert l\right\vert }\left(  A_{\Sigma}\right)  \label{4}%
\end{equation}

\end{proof}

Notice, that in the proof of the above Proposition, we assumed $\rho_{\alpha
}\neq0$ for all $\alpha$, this allowed us to defined the constants
$\theta_{\alpha}$. The assumption $\rho_{\alpha}\neq0$ is not necessary. Since
if there exists $\alpha\in\left\{  1,..,q\right\}  $ such that $\rho_{\alpha
}=0$. then (see \cite{6})
\begin{align*}
\tilde{\sigma}_{u}  &  =\tilde{\sigma}_{u}(\rho_{1}A_{\Sigma}+\mu_{1}%
I,\ldots,\rho_{q}A_{\Sigma}+\mu_{q}I)\\
&  =\tilde{\sigma}_{u}(\rho_{1}A_{\Sigma}+\mu_{1}I,\ldots,\rho_{i-1}A_{\Sigma
}+\mu_{i-i}I,\mu_{i}I,\rho_{i+1}A_{\Sigma}+\mu_{i+1}I,\ldots,\rho_{q}%
A_{\Sigma}+\mu_{q}I)\\
&  =\mu_{i}^{u_{i}}\tilde{\sigma}_{\tilde{u}}(\rho_{1}A_{\Sigma}+\mu
_{1}I,\ldots,\rho_{i-1}A_{\Sigma}+\mu_{i-i}I,\rho_{i+1}A_{\Sigma}+\mu
_{i+1}I,\ldots,\rho_{q}A_{\Sigma}+\mu_{q}I),
\end{align*}
where%
\[
\tilde{u}=\left(  u_{1},\ldots,u_{i-1},u_{i+1},\ldots,u_{q}\right)
\]
and we may apply the above Proposition to $\tilde{\sigma}_{\tilde{u}}$.

\begin{proposition}
\label{Prop2} Let $\tilde{A}=(A_{1}|_{\Sigma},...,A_{q}|_{\Sigma})$ ,where
$A_{\alpha}|_{\Sigma}=(\rho_{\alpha}A_{\Sigma}+\mu_{\alpha}I$, and
$\overline{A}=(A_{\Sigma},\mu_{1}I,...,\mu_{q}I).$ For any multi-index $u\in%
\mathbb{N}
^{q}$, put $\tilde{\sigma}_{u}=\sigma_{u}(\tilde{A})$ and $\overline{\sigma
}_{u}=\sigma_{u}\left(  \overline{A}\right)  $. Then we have%
\begin{equation}
\tilde{\sigma}_{u}=%
{\displaystyle\sum\limits_{l\leq u}}
\left(
\begin{array}
[c]{c}%
\left\vert l\right\vert \\
l
\end{array}
\right)  \rho^{l}\overline{\sigma}_{\left(  |l|,u-l\right)  }. \label{7}%
\end{equation}

\end{proposition}

\begin{proof}
First notice that
\[
\overline{\sigma}_{\left(  j,v\right)  }=\sigma_{\left(  j,v\right)  }(\bar
{A})=\mu^{v}\sigma_{\left(  j,v\right)  }(A_{\Sigma},I,\ldots,I).
\]
Using the following formula (see \cite{6})
\[
\sigma_{(}v,j)(B_{1},\ldots,B_{q},I)=\binom{n-1-|v|}{j}\sigma_{v}(B_{1}%
,\ldots,B_{q})
\]
consecutively $q$ -times, we obtain
\begin{align*}
\sigma_{(j,v)}(A_{\Sigma},I,\ldots,I)  &  =\binom{n-1-|v|+v_{1}}{v_{1}}%
\ldots\binom{n-1-|v|+|v|}{v_{q}}\sigma_{j}(A_{\Sigma})\\
&  =\frac{1}{\left(  n-1-|v|-j\right)  !}\binom{n-1-j}{v}\sigma_{j}(A_{\Sigma
})
\end{align*}
or by putting $l=u-v$%
\[
\sigma_{\left(  |l|,u-l\right)  }=\frac{1}{\left(  n-1-\left\vert u\right\vert
\right)  !}\left(
\begin{array}
[c]{c}%
n-1-\left\vert l\right\vert \\
u-l
\end{array}
\right)  \sigma_{\left\vert l\right\vert }(A_{\Sigma})
\]
so%
\[
\overline{\sigma}_{\left(  |l|,u-l\right)  }=\frac{\mu^{u-l}}{\left(
n-1-|u|\right)  !}\left(
\begin{array}
[c]{c}%
n-1-\left\vert l\right\vert \\
u-l
\end{array}
\right)  \sigma_{\left\vert l\right\vert }(A_{\Sigma})
\]
Thus by (\ref{1'}) we get (\ref{7}).
\end{proof}

Let us now state the relation between symmetric functions corresponding to the
families $(A_{\alpha})$ and $(A_{\alpha}|\Sigma)$. Notice that there are of
different sizes.

\begin{proposition}
\label{Prop3} Let $A=(A_{1},...,A_{q})$ and $\tilde{A}=(A_{1}|_{\Sigma
},...,A_{q}|_{\Sigma})$. Put $\sigma_{u}=\sigma_{u}(A)$ and $\tilde{\sigma
}_{u}=\sigma_{u}(\tilde{A})$. Then
\begin{multline*}
\sigma_{u}=\tilde{\sigma}_{u}+\underset{\alpha}{\sum}C_{\alpha}\tilde{\sigma
}_{\alpha_{\flat}(v)}\\
+\underset{\alpha^{\sharp}\beta^{\sharp}(0)\leq w\leq u}{\sum}\underset
{\alpha,\beta}{\sum}(-1)^{\left\vert w\right\vert -\left\vert v\right\vert
+1}\binom{\left\vert u\right\vert -\left\vert w\right\vert }{u-w}B_{\alpha
}^{\top}\tilde{A}^{u-w}B_{\beta}\tilde{\sigma}_{\alpha_{\flat}\beta_{\flat
}(w)}%
\end{multline*}
where $\ B_{\alpha}^{\top}=\left(  \langle A_{\alpha}v,e_{1}\rangle
,...,\langle A_{\alpha}v,e_{n-1}\rangle\right)  $ and $C_{\alpha}=\langle
A_{\alpha}v,v\rangle$.
\end{proposition}

\begin{proof}
First calculate the characteristic polynomial $P_{A}(t)$ of $A=(A_{1}%
,...,A_{q})$. By definition, we have
\begin{align*}
P_{A}(t)  &  =\det(I+\underset{\alpha}{\sum}t_{\alpha}A_{\alpha})\\
&  =\det\left(  I_{n-1}+\underset{\alpha}{\sum}t_{\alpha}A_{\alpha}|_{\Sigma
}\right)  \det\left(  C-B^{\top}\left(  I_{n-1}+\underset{\alpha}{\sum
}t_{\alpha}A_{\alpha}|_{\Sigma}\right)  ^{-1}B\right)
\end{align*}
where $B^{\top}$ and $C$ are respectively given by
\[
B^{\top}=\left(  \underset{\alpha}{\sum}t_{\alpha}\langle A_{\alpha}%
v,e_{1}\rangle,...,\underset{\alpha}{\sum}t_{\alpha}\langle A_{\alpha
}v,e_{n-1}\rangle\right)  =\sum_{\alpha}t_{\alpha}B_{\alpha}^{\top}%
\]
and
\[
C=1+\underset{\alpha}{\sum}t_{\alpha}\langle A_{\alpha}v,v\rangle\text{.}%
\]
To simplify the expressions, we put%
\[
M=\left(  I_{n-1}+\underset{\alpha}{\sum}t_{\alpha}A_{\alpha}|_{\Sigma
}\right)  .
\]
and%
\[
f(t)=\det(C-B^{\top}M^{-1}B).
\]
Moreover it is clear that%
\[
B^{\top}M^{-1}B=\underset{\alpha,\beta}{\sum}t_{\alpha}t_{\beta}B_{\alpha
}^{\top}M^{-1}B_{\beta}\text{ ,}%
\]
therefore we get%
\[
f(t)=\left(  1+\underset{\alpha}{\sum}t_{\alpha}C_{\alpha}-\underset
{\alpha,\beta}{\sum}t_{\alpha}t_{\beta}B_{\alpha}^{\top}M^{-1}B_{\beta
}\right)  .
\]
As it is easily seen that
\[
P_{A}(t)=f(t)P_{\tilde{A}}(t).
\]
where $P_{A}(t)$ and $P_{\tilde{A}}(t)$ are defined by
\[
P_{A}(t)=\underset{\left\vert u\right\vert \leq n}{\sum}\sigma_{u}t^{u}%
\]
and%
\[
P_{\tilde{A}}(t)=\underset{\left\vert u\right\vert \leq n-1}{\sum}%
\tilde{\sigma}_{u}t^{u}.
\]
Thus we get%
\[
\frac{\partial^{u}}{\partial t^{u}}P_{A}(t)\mid_{t=0}=u!\sigma_{u},\quad
\frac{\partial^{u}}{\partial t^{u}}P_{\tilde{A}}(t)\mid_{t=0}=u!\tilde{\sigma
}_{u}.
\]
Similarly we obtain%
\[
\frac{\partial^{u}}{\partial t^{u}}\left(  P_{\tilde{A}}(t).f(t)\right)
\mid_{t=0}=\underset{v\leq u}{\sum}\binom{u}{v}\left(  \frac{\partial^{v}%
}{\partial t^{v}}P_{\tilde{A}}(t).\frac{\partial^{u-v}}{\partial t^{u-v}%
}f(t)\right)  \mid_{t=0},
\]
where
\[
\binom{u}{v}=\frac{u!}{v!}=\frac{u_{1}!\ldots u_{1}!}{v_{1}!\ldots v_{q}!}.
\]
Now we will compute $\frac{\partial^{u-v}}{\partial t^{u-v}}f(t)$. For $v=u, $
it is checked%
\[
\frac{\partial^{u-v}}{\partial t^{u-v}}f(t)\mid_{t=0}=f(0,...,0)=1\text{.}%
\]
If $v=\alpha_{\flat}(u)$, we have%
\[
\frac{\partial^{u-v}}{\partial t^{u-v}}f(t)\mid_{t=0}=\frac{\partial}{\partial
t_{\alpha}}f(t)\mid_{t=0}=C_{\alpha}%
\]
and for any $v\leq\alpha_{\flat}\beta_{\flat}(u),$ we obtain%
\[
\frac{\partial^{u-v}}{\partial t^{u-v}}f(t)\mid_{t=0}=\frac{\partial^{u-v}%
}{\partial t^{u-v}}\left(  1+\underset{\alpha}{\sum}t_{\alpha}C_{\alpha
}-\underset{\alpha,\beta}{\sum}t_{\alpha}t_{\beta}B_{\alpha}^{\top}\left(
I_{n-1}+\underset{\alpha}{\sum}t_{\alpha}A_{\alpha}|_{\Sigma}\right)
^{-1}B_{\beta}\right)  \mid_{t=0}.
\]
Taking $\left\vert t\right\vert <\varepsilon$ for some small enough
$\varepsilon$ we get
\[
\left\Vert \underset{\alpha}{\sum}t_{\alpha}A_{\alpha}|_{\Sigma}\right\Vert
\leq1.
\]
Thus
\[
\left(  I_{n-1}+\underset{\alpha}{\sum}t_{\alpha}A_{\alpha}|_{\Sigma}\right)
^{-1}=\underset{k=0}{\overset{\infty}{\sum}}(-1)^{k}\left(  \underset{\alpha
}{\sum}t_{\alpha}A_{\alpha}|_{\Sigma}\right)  ^{k}%
\]
Since the matrices $A_{\alpha}|_{\Sigma}$ are diagonal, then they commute.
Therefore, we may write
\[
\left(  \underset{\alpha}{\sum}t_{\alpha}A_{\alpha}|_{\Sigma}\right)
^{k}=\underset{w_{1}+...+w_{q}=k}{\sum}\binom{k}{w}t^{w}\tilde{A}^{w}.
\]
Hence%
\[
\left(  I_{n-1}+\underset{\alpha}{\sum}t_{\alpha}A_{\alpha}|_{\Sigma}\right)
^{-1}=\underset{k=0}{\overset{\infty}{\sum}}(-1)^{k}\left(  \underset
{w_{1}+...+w_{q}=k}{\sum}\binom{k}{w}t^{w}\tilde{A}^{w}\right)
\]
Which gives%

\[
\frac{\partial^{u-v}}{\partial t^{u-v}}f(t)\mid_{t=0}=\frac{\partial^{u-v}%
}{\partial t^{u-v}}\left(  -\underset{\alpha,\beta}{\sum}t_{\alpha}t_{\beta
}B_{\alpha}^{\top}\underset{k=0}{\overset{\infty}{\sum}}(-1)^{k}\left(
\underset{w_{1}+...+w_{q}=\left\vert w\right\vert =k}{\sum}\binom{\left\vert
w\right\vert }{w}t^{w}\tilde{A}^{w}\right)  B_{\beta}\right)  \mid_{t=0}%
\]%
\[
\ \ \ \ \ \ \ \ \ \ \ \ \ \ \ \ \ \ \ \ \ \ \ =\frac{\partial^{u-v}}{\partial
t^{u-v}}\left(  -\underset{\alpha,\beta}{\sum}B_{\alpha}^{\top}\underset
{k=0}{\overset{\infty}{\sum}}(-1)^{k}\left(  \underset{w_{1}+...+w_{q}%
=\left\vert w\right\vert =k}{\sum}\binom{\left\vert w\right\vert }%
{w}t^{w+\alpha^{\sharp}\beta^{\sharp}(0)}\tilde{A}^{w}\right)  B_{\beta
}\right)  \mid_{t=0}%
\]%
\[
\ \ \ \ \ \ \ \ \ \ \ \ \ \ \ \ \ \ \ \ \ \ \ \ =-\underset{\alpha,\beta}%
{\sum}\sum_{v\leq\alpha_{\flat}\beta_{\flat}(u)}B_{\alpha}^{\top
}(-1)^{\left\vert u\right\vert -\left\vert v\right\vert -2}(u-v)!\binom
{\left\vert u\right\vert -\left\vert v\right\vert -2}{u-\alpha^{\sharp}%
\beta^{\sharp}(v)}\tilde{A}^{u-\alpha^{\sharp}\beta^{\sharp}(v)}B_{\beta}%
\]%
\[
\ \ \ \ \ \ \ \ \ \ \ \ \ \ \ \ \ \ \ \ \ \ =\underset{\alpha,\beta}{\sum}%
\sum_{v\leq\alpha_{\flat}\beta_{\flat}(u)}B_{\alpha}^{\top}(-1)^{\left\vert
u\right\vert -\left\vert v\right\vert -1}(u-v)!\binom{\left\vert u\right\vert
-\left\vert v\right\vert -2}{u-\alpha^{\sharp}\beta^{\sharp}(v)}\tilde
{A}^{u-\alpha^{\sharp}\beta^{\sharp}(v)}B_{\beta}\text{ .}%
\]
Finally we obtain
\begin{align*}
u!\sigma_{u}  &  =u!\tilde{\sigma}_{u}+\underset{\alpha}{\sum}\left(
\alpha_{\flat}(u)\right)  !\binom{u}{\alpha_{\flat}(u)}C_{\alpha}\tilde
{\sigma}_{\alpha_{\flat}(u)}\\
&  +\underset{\alpha,\beta}{\sum}\underset{v\leq u}{\sum}\binom{u}%
{v}(-1)^{\left\vert u\right\vert -\left\vert v\right\vert -1}(u-v)!v!\binom
{\left\vert u\right\vert -\left\vert v\right\vert -2}{u-\alpha_{\flat}%
\beta_{\flat}(v)}B_{\alpha}^{\top}\tilde{A}^{u-\alpha_{\flat}\beta_{\flat}%
(v)}B_{\beta}\tilde{\sigma}_{v}%
\end{align*}
or equivalently,
\[
\sigma_{u}=\tilde{\sigma}_{u}+\underset{\alpha}{\sum}C_{\alpha}\tilde{\sigma
}_{\alpha_{\flat}(u)}+\underset{\alpha,\beta}{\sum}\underset{0\leq v\leq
\alpha_{\flat}\beta_{\flat}(u)}{\sum}(-1)^{\left\vert u\right\vert -\left\vert
v\right\vert -1}\binom{\left\vert u\right\vert -\left\vert v\right\vert
-2}{u-\alpha^{\sharp}\beta^{\sharp}(v)}B_{\alpha}^{\top}\tilde{A}%
^{u-\alpha^{\sharp}\beta^{\sharp}(v)}B_{\beta}\tilde{\sigma}_{v}.
\]
Hence
\[
\sigma_{u}=\tilde{\sigma}_{u}+\underset{\alpha}{\sum}C_{\alpha}\tilde{\sigma
}_{\alpha_{\flat}(u)}+\underset{\alpha^{\sharp}\beta^{\sharp}(0)\leq w\leq
u}{\sum}\underset{\alpha,\beta}{\sum}(-1)^{\left\vert w\right\vert -\left\vert
v\right\vert +1}\binom{\left\vert u\right\vert -\left\vert w\right\vert }%
{u-w}B_{\alpha}^{\top}\tilde{A}^{u-w}B_{\beta}\tilde{\sigma}_{\alpha_{\flat
}\beta_{\flat}(w)}\text{.}%
\]

\end{proof}

\section{Generalized Newton transformation on the boundary}

We use the same notations as in previous sections. In this section we give the
expression of the GNT $T_{u}=T_{u}(\tilde{A})$, where $\tilde{A}=(A_{1}%
|\Sigma,\ldots,A_{q}|\Sigma)$, on the boundary $\Sigma^{n-1}$of $M^{n} $.
Recall that
\[
A_{\alpha}|_{\Sigma}=\rho_{\alpha}A_{\Sigma}+\mu_{\alpha}I,
\]
where
\[
\rho_{\alpha}=\langle\eta,N_{\alpha}\rangle,\quad\mu_{\alpha}=\langle
V,N_{\alpha}\rangle,\quad V=\sum_{\alpha=1}^{q}\lambda_{\alpha}\xi_{\alpha}.
\]

\begin{proposition}
\label{Prop4} Let $\overline{M}^{n+q}$ be an $(n+q)$-Riemannian manifold and
$P^{n}\subset\overline{M}^{n+q}$ an oriented totally umbilical $n$-submanifold
of $\overline{M}^{n+q}$. Denote by $\Sigma^{n-1}\subset P^{n}$ an
$(n-1)$-compact hypersurface of $P^{n}$. Let $\Psi:M^{n}\longrightarrow
\overline{M}^{n+q}$ be an oriented connected and compact submanifold of
$\overline{M}^{n+q}$ with boundary $\Sigma^{n-1}=\Psi\left(  \partial
M\right)  .$ Then along the boundary $\partial M$, we have%
\begin{equation}
\left\langle T_{u}\nu,\nu\right\rangle =\tilde{\sigma}_{u}(A_{1}|_{\Sigma
},...,A_{q}|_{\Sigma}). \label{9}%
\end{equation}

\end{proposition}

\begin{proof}
We make a recursive proof. Assume that (\ref{9}) holds for any multi--index
$v<u$. We have by the recurrence definition of $T_{u}$%
\begin{align*}
\left\langle T_{u}\nu,\nu\right\rangle  &  =\sigma_{u}\left\langle \nu
,\nu\right\rangle -\underset{\alpha}{\sum}\left\langle A_{\alpha}%
T_{\alpha_{\flat}(u)}\nu,\nu\right\rangle \\
&  =\sigma_{u}-\underset{\alpha}{\sum}\left\langle T_{\alpha_{\flat}(u)}%
\nu,A_{\alpha}\nu\right\rangle \\
&  =\sigma_{u}-\underset{\alpha}{\sum}\left\langle T_{\alpha_{\flat}(u)}%
\nu,\nu\right\rangle \left\langle A_{\alpha}\nu,\nu\right\rangle
-\underset{\alpha,i}{\sum}\left\langle T_{\alpha_{\flat}(u)}\nu,e_{i}%
\right\rangle \left\langle A_{\alpha}e_{i},\nu\right\rangle .
\end{align*}
Put
\[
C_{\alpha}=\left\langle A_{\alpha}\nu,\nu\right\rangle \text{ ,\ }b_{i,\alpha
}=\left\langle A_{\alpha}e_{i},\nu\right\rangle
\]
then%

\[
\left\langle T_{u}\nu,\nu\right\rangle =\sigma_{u}-\underset{\alpha}{\sum
}C_{\alpha}\tilde{\sigma}_{\alpha_{\flat}(u)}-\underset{\alpha,i}{\sum
}b_{i,\alpha}\left\langle T_{\alpha_{\flat}(u)}\nu,e_{i}\right\rangle .
\]
Let us compute $\left\langle T_{u}\nu,e_{i}\right\rangle $ for any
multi--index $u$ $\in%
\mathbb{N}
^{q}.$ Notice that%
\[
A_{\alpha}e_{i}=\rho_{\alpha}A_{\Sigma}e_{i}+\mu_{\alpha}e_{i}+b_{i,\alpha}\nu
\]
assuming that $\left\{  e_{1},...,e_{n-1}\right\}  $ is an orthonormal basis
of $T_{p}\Sigma^{n-1}$ consisting of eigenvectors of $A_{\Sigma}$ (with
eigenvalues $\tau_{i}$). Thus
\[
A_{\alpha}e_{i}=\gamma_{i,\alpha}e_{i}+b_{i,\alpha}\nu
\]
where%
\[
\gamma_{i,\alpha}=\rho_{\alpha}\tau_{i}+\mu_{\alpha}.
\]
We will show inductively that%
\begin{equation}
\left\langle T_{u}\nu,e_{i}\right\rangle =\underset{\alpha}{\sum}%
\underset{\alpha^{\sharp}(0)\leq w\leq u}{\sum}(-1)^{\left\vert u\right\vert
-\left\vert w\right\vert +1}\binom{\left\vert u\right\vert -\left\vert
w\right\vert }{u-w}b_{i,\alpha}\gamma_{i}^{u-w}\tilde{\sigma}_{\alpha_{\flat
}(w)} \label{10}%
\end{equation}
where%
\[
\gamma_{i}=\left(  \gamma_{i,1},...,\gamma_{i,q}\right)
\]
Indeed, for $u=\beta^{\sharp}(0)$ we have,%
\begin{align*}
\left\langle T_{u}\nu,e_{i}\right\rangle  &  =\sigma_{\beta^{\sharp}%
(0)}\left\langle \nu,e_{i}\right\rangle -\underset{\alpha}{\sum}\left\langle
A_{\alpha}T_{\alpha_{\flat}(\beta^{\sharp}(0))}\nu,\nu\right\rangle \\
&  =-\underset{\alpha}{\sum}\left\langle A_{\alpha}\nu,\nu\right\rangle \\
&  =-b_{i,\beta}\\
&  =\underset{\alpha}{\sum}\underset{\alpha^{\sharp}(0)\leq w\leq\beta
^{\sharp}(0)}{\sum}(-1)^{\left\vert \beta^{\sharp}(0)\right\vert -\left\vert
w\right\vert +1}\binom{\left\vert \beta^{\sharp}(0)\right\vert -\left\vert
w\right\vert }{\beta^{\sharp}(0)-w}b_{i,\alpha}\gamma_{i}^{\beta^{\sharp
}(0)-w}\tilde{\sigma}_{\alpha_{\flat}(w)},
\end{align*}
since the sum reduces to one element for $\alpha=\beta$.

Assume that (\ref{10}) holds for all multi index $v<u$. Then again by the
recursive definition of $T_{u}$ we get%
\begin{align*}
\left\langle T_{u}\nu,e_{i}\right\rangle  &  =\sigma_{u}\left\langle \nu
,e_{i}\right\rangle -\underset{\alpha}{\sum}\left\langle A_{\alpha}%
T_{\alpha_{\flat}(u)}\nu,e_{i}\right\rangle \\
&  =-\underset{\alpha}{\sum}\left\langle T_{\alpha_{\flat}(u)}\nu,A_{\alpha
}e_{i}\right\rangle \\
&  =-\underset{\alpha}{\sum}\left\langle T_{\alpha_{\flat}(u)}\nu
,\nu\right\rangle \left\langle A_{\alpha}e_{i},\nu\right\rangle -\underset
{\alpha,j}{\sum}\left\langle T_{\alpha_{\flat}(u)}\nu,e_{i}\right\rangle
\left\langle A_{\alpha}e_{i},e_{j}\right\rangle \\
&  =-\underset{\alpha}{\sum}\left\langle T_{\alpha_{\flat}(u)}\nu
,e_{i}\right\rangle \gamma_{i,\alpha}-\underset{\alpha}{\sum}b_{i,\alpha
}\tilde{\sigma}_{\alpha_{\flat}(u)}\\
&  =-\underset{\alpha}{\sum}\gamma_{i,\alpha}\left(  \underset{\beta}{\sum
}\underset{\beta^{\sharp}(0)\leq w\leq\alpha_{\flat}(u)}{\sum}(-1)^{\left\vert
u\right\vert -\left\vert w\right\vert }\binom{\left\vert u\right\vert
-\left\vert w\right\vert -1}{\alpha_{\flat}(u)-w}b_{i,\beta}\gamma_{i}%
^{\alpha_{\flat}(u)-w}\tilde{\sigma}_{\beta_{\flat}(w)}\right) \\
&  -\underset{\alpha}{\sum}b_{i,\alpha}\tilde{\sigma}_{\alpha_{\flat}(u)}.
\end{align*}
Clearly%
\[
\gamma_{i,\alpha}.\gamma_{i}^{\alpha_{\flat}(u)-w}=\gamma_{i}^{u-w}%
\]
Notice that taking $w=u$ we get the last sum. Thus we obtain%
\[
\left\langle T_{u}\nu,e_{i}\right\rangle =\underset{\alpha,\beta}{\sum}\left(
\underset{\beta^{\sharp}(0)\leq w\leq u}{\sum}(-1)^{\left\vert u\right\vert
-\left\vert w\right\vert +1}\binom{\left\vert u\right\vert -\left\vert
w\right\vert -1}{\alpha_{\flat}(u)-w}b_{i,\beta}\gamma_{i}^{u-w}\tilde{\sigma
}_{\beta_{\flat}(w)}\right)  .
\]
Since%
\[
\underset{\alpha}{\sum}\binom{\left\vert u\right\vert -\left\vert w\right\vert
-1}{\alpha_{\flat}(u)-w}=\binom{\left\vert u\right\vert -\left\vert
w\right\vert }{u-w}%
\]
we deduce
\[
\left\langle T_{u}\nu,e_{i}\right\rangle =\underset{\beta}{\sum}\left(
\underset{\beta^{\sharp}(0)\leq w\leq u}{\sum}(-1)^{\left\vert u\right\vert
-\left\vert w\right\vert +1}\binom{\left\vert u\right\vert -\left\vert
w\right\vert }{u-w}b_{i,\beta}\gamma_{i}^{u-w}\tilde{\sigma}_{\beta_{\flat
}(w)}\right)
\]
which end the proof of (\ref{10}).

Now, we may prove (\ref{9}). First, we have
\[
\underset{\alpha,i}{\sum}\left\langle T_{\alpha_{\flat}(u)}\nu,e_{i}%
\right\rangle b_{i,\alpha}=\underset{\alpha,\beta,i}{\sum}\left(
\underset{\beta^{\sharp}(0)\leq w\leq\alpha_{\flat}(u)}{\sum}(-1)^{\left\vert
u\right\vert -\left\vert w\right\vert }\binom{\left\vert u\right\vert
-\left\vert w\right\vert -1}{\alpha_{\flat}(u)-w}b_{i,\beta}\gamma_{i}%
^{\alpha_{\flat}(u)-w}\tilde{\sigma}_{\beta_{\flat}(w)}\right)
\]
Replacing $w$ by $\alpha_{\flat}\beta_{\flat}(w)$, we obtain
\[
\underset{\alpha,i}{\sum}\left\langle T_{\alpha_{\flat}(u)}\nu,e_{i}%
\right\rangle b_{i,\alpha}=\underset{\alpha,\beta,i}{\sum}\left(
\underset{\alpha^{\sharp}\beta^{\sharp}(0)\leq w\leq u}{\sum}(-1)^{\left\vert
u\right\vert -\left\vert w\right\vert +1}\binom{\left\vert u\right\vert
-\left\vert w\right\vert }{u-w}b_{i,\alpha}\gamma_{i}^{u-w}b_{i,\beta}%
\tilde{\sigma}_{\alpha_{\flat}\beta_{\flat}(w)}\right)
\]
Noticing that
\[
\underset{i}{\sum}b_{i,\alpha}\gamma_{i}^{u-w}b_{i,\beta}=B_{\alpha}^{\top
}\tilde{A}^{u-w}B_{\beta}%
\]
we infer%
\[
\underset{\alpha,i}{\sum}\left\langle T_{\alpha_{\flat}(v)}\nu,e_{i}%
\right\rangle b_{i,\alpha}=\underset{\alpha,\beta}{\sum}\underset
{\alpha^{\sharp}\beta^{\sharp}(0)\leq w\leq u}{\sum}(-1)^{\left\vert
u\right\vert -\left\vert w\right\vert +1}\binom{\left\vert u\right\vert
-\left\vert w\right\vert }{u-w}B_{\alpha}^{\top}\tilde{A}^{u-w}B_{\beta
}.\tilde{\sigma}_{\alpha_{\flat}\beta_{\flat}(w)}%
\]
Summing up all the above considerations we get%
\begin{align*}
\left\langle T_{u}\nu,\nu\right\rangle  &  =\sigma_{u}-\underset{\alpha}{\sum
}\left\langle T_{\alpha_{\flat}(u)}\nu,\nu\right\rangle \left\langle
A_{\alpha}\nu,\nu\right\rangle -\underset{\alpha,i}{\sum}\left\langle
T_{\alpha_{\flat}(u)}\nu,e_{i}\right\rangle \left\langle A_{\alpha}e_{i}%
,\nu\right\rangle \\
&  =\sigma_{u}-\sum C_{\alpha}\tilde{\sigma}_{\alpha_{\flat}(u)}%
-\underset{\alpha,\beta}{\sum}\underset{\alpha^{\sharp}\beta^{\sharp}(0)\leq
w\leq u}{\sum}(-1)^{\left\vert u\right\vert -\left\vert w\right\vert +1}%
\binom{\left\vert u\right\vert -\left\vert w\right\vert }{u-w}B_{\alpha}%
^{\top}\tilde{A}^{u-w}B_{\beta}.\tilde{\sigma}_{\alpha_{\flat}\beta_{\flat
}(w)}%
\end{align*}
or, equivalently,
\[
\sigma_{u}=\tilde{\sigma}_{u}+\underset{\alpha}{\sum}C_{\alpha}\tilde{\sigma
}_{\alpha_{\flat}(u)}+\underset{\alpha,\beta}{\sum}\underset{\alpha^{\sharp
}\beta^{\sharp}(0)\leq w\leq u}{\sum}(-1)^{\left\vert u\right\vert -\left\vert
w\right\vert +1}\binom{\left\vert u\right\vert -\left\vert w\right\vert }%
{u-w}B_{\alpha}^{\top}\tilde{A}^{u-w}B_{\beta}.\tilde{\sigma}_{\alpha_{\flat
}\beta_{\flat}(w)}%
\]
Applying Proposition (\ref{Prop3}), we obtain
\[
\left\langle T_{u}\nu,\nu\right\rangle =\tilde{\sigma}_{u}.
\]

\end{proof}

By Proposition (\ref{Prop1}) we obtain the following expression of $\langle
T_{u}\nu,\nu\rangle$ in terms of symmetric functions of the shape operator
$A_{\Sigma}$.

\begin{corollary}
\label{Cor1} With the conditions of Proposition (\ref{Prop4}), we have
\begin{equation}
\langle T_{u}\nu,\nu\rangle=\frac{1}{n-1-|u|}\sum_{l\leq u}\binom
{n-1-|l|}{\left\vert u\right\vert -l}\rho^{l}\mu^{u-l}\sigma_{|l|}(A_{\Sigma
}). \label{11}%
\end{equation}

\end{corollary}

The relation (\ref{11}) will be more simple if we suppose that the embedding
$P^{n}\subset\overline{M}^{n+q}$ is totally geodesic.

\begin{corollary}
\label{Cor2} With the conditions of Proposition (\ref{Prop4}) and assuming
that $P^{n}\subset\overline{M}^{n+q}$\ is totally geosedic, then for every
multi--index $u$ with length $\left\vert u\right\vert \leq n-1$, we have%
\[
\left\langle T_{u}\nu,\nu\right\rangle =\rho^{u}\sigma_{|u|}\left(  A_{\Sigma
}\right)
\]

\end{corollary}

\begin{proof}
It suffices to use (\ref{11}) with $\mu_{\alpha}=0$.
\end{proof}

\section{Transversality of submanifolds}

The formula for the generalized Newton transformation implies the relation
between transversality of $M^{n}$ and $P^{n}$ and ellipticity of $T_{u}$
provided that $P^{n}$ is totally geodesic in $\overline{M}^{n+1}$. This
generalizes the result in (\cite{2}) to any arbitrary codimension.

\begin{theorem}
With the conditions in Corollary (\ref{Cor2}) the submanifolds $M^{n}$ and
$P^{n}$ are transversal along $\partial M$ provided that for some multi--index
$u$ of length $1\leq\left\vert u\right\vert \leq n-1,$ the generalized Newton
transformation $T_{u}$ is positive definite on $M^{n\text{ }}$.
\end{theorem}

\begin{proof}
Saying that $M^{n\text{ }}$and $P^{n}$ are not transversal means that there
exist $p\in\partial M$ such that for every $\alpha\in\left\{  1,...,q\right\}
$ we have%
\[
\rho_{u}=\left\langle \eta,N_{\alpha}\right\rangle =0\quad\text{at $p$}.
\]
Therefore, if we suppose that for all $p\in\overline{M}^{n+q}$, $T_{u}$ is
positive definite, then by Corollary 7, $\rho^{u}(p)\neq0$. Thus
\[
\left\langle \eta,N_{\alpha}\right\rangle \neq0,
\]
hence $M^{n\text{ }}$and $P^{n}$ are transversal.
\end{proof}

\end{document}